\documentclass[12pt]{amsart}
\usepackage{amssymb,latexsym,comment,url}

\newcommand{\Aff}{{\mathbb A}}
\newcommand{\Aut}{{\operatorname{Aut}}}
\newcommand{\charic}{{\operatorname{char}\,}} 
\newcommand{\FF}{{\mathcal F}}
\newcommand{\GG}{{\mathcal G}}
\newcommand{\GL}{\operatorname{GL}}
\newcommand{\Hom}{\operatorname{Hom}}
\newcommand{\id}{\operatorname{id}}
\newcommand{\isom}{{\, \cong \,}}
\newcommand{\KK}{{\mathbb K}}
\newcommand{\LL}{{\mathbb L}}
\newcommand{\CL}{{\mathcal L}}
\newcommand{\la}{{\lambda}}
\newcommand{\ord}{{\operatorname{ord}}}
\newcommand{\PP}{{\mathbb P}}
\newcommand{\Pic}{{\operatorname{Pic}}}
\newcommand{\Q}{{\mathbb Q}}
\newcommand{\ra}{{\, \longrightarrow \,}}
\newcommand{\ratto}{{\,- \to\,}}
\newcommand{\SL}{\operatorname{SL}}
\newcommand{\liesl}{\mathfrak{sl}}
\newcommand{\Sec}{{\operatorname{Sec}\,}}
\newcommand{\Tan}{\operatorname{Tan}}
\newcommand{\UU}{{\mathbb U}}
\newcommand{\VV}{{\mathbb V}}
\newcommand{\Z}{{\mathbb Z}}

\newcommand{\MatP}{\left({\begin{array}{c@{\quad}c@{\,\,\,}c@{\,\,\,}cc}
    1   &    0   &    0    & \cdots &      0     \\
    0   & \zeta_n  &    0    & \cdots &      0     \\
    0   &    0   & \zeta_n^2 & \cdots &      0     \\ 
 \vdots & \vdots & \vdots  &        &   \vdots   \\
    0   &   0    &    0    & \cdots & \zeta_n^{n-1}  
\end{array}}\right)}

\newcommand{\MatQ}{\left({\begin{array}{c@{\quad}c@{\quad}c@{\quad}c@{\quad}c}
    0   &    0   & \cdots &   0    &      1     \\ 
    1   &    0   & \cdots &   0    &      0     \\
    0   &    1   & \cdots &   0    &      0     \\ 
 \vdots & \vdots &       & \vdots &   \vdots   \\
    0   &   0    & \cdots &   1    &      0       
\end{array}}\right)}

\newenvironment{ProofOf}[1]{\par\noindent{\it Proof of #1.}}%
                       {\hspace*{\fill}\nobreak$\qed$\par\medskip}

\newtheorem{thm}{Theorem}[section]
\newtheorem{prop}[thm]{Proposition}
\newtheorem{lem}[thm]{Lemma}

\theoremstyle{definition}

\newtheorem{defn}[thm]{Definition}
\newtheorem{rem}[thm]{Remark}
\newtheorem{example}[thm]{Example}

\renewcommand{\theenumi}{\roman{enumi}}

\begin{document}

\title[Invariant theory for the elliptic normal quintic]%
{Invariant theory for the~elliptic~normal~quintic,\\
II. The covering map}

\author{Tom Fisher}
\address{University of Cambridge,
         DPMMS, Centre for Mathematical Sciences,
         Wilberforce Road, Cambridge CB3 0WB, UK}
\email{T.A.Fisher@dpmms.cam.ac.uk}

\date{11th March 2013}  

\begin{abstract}
  A genus one curve $C$ of degree $5$ is defined by the $4 \times 4$
  Pfaffians of a $5 \times 5$ alternating matrix of linear forms on
  $\PP^4$. We prove a result characterising the covariants for these
  models in terms of their restrictions to the family of curves
  parametrised by the modular curve $X(5)$.  We then construct
  covariants describing the covering map of degree $25$ from $C$ to
  its Jacobian and give a practical algorithm for evaluating them.
\end{abstract}

\maketitle


\section{Introduction}
\label{sec:intro}

\begin{defn}
\label{defdegen}
Let $n \ge 3$ be an integer.
\begin{enumerate}
\item An {\em elliptic normal curve} $C \subset \PP^{n-1}$ is a smooth
  curve of genus one and degree $n$ that spans $\PP^{n-1}$.
\item A {\em rational nodal curve} $C \subset \PP^{n-1}$ is a rational
  curve of degree $n$ that spans $\PP^{n-1}$ and has a single node.
\end{enumerate}
\end{defn}

If $C \subset \PP^{n-1}$ is an elliptic normal curve then there is a
covering map $\pi$ of degree $n^2$ from $C$ to its Jacobian $E$ given
by $P \mapsto [nP-H] \in \Pic^0(C) \isom E$ where $H$ is the
hyperplane section. We may also describe $\pi : C \to E$ as the map
that quotients out by the action of $E[n]$ on $C$ by translation
(assuming we are not in characteristic dividing~$n$).  The subgroup of
$\SL_n$ consisting of matrices that describe this action is called the
{\em Heisenberg group} of $C$. If $n$ is odd then over an
algebraically closed field we may change co-ordinates so that this
group is generated by
\begin{equation}
\label{heismats}
\MatP \quad {\mbox{ and }} \quad \MatQ 
\end{equation}
where $\zeta_n$ is a primitive $n$th root of unity. (If $n$ is even
then one must take scalar multiples of these matrices with determinant
$1$.)

In the cases $n=2,3,4$ classical invariant theory gives formulae for
the Jacobian $E$ and for the covering map $\pi: C \to E$.  See
\cite{We1}, \cite{We2} for the cases $n=2,3$, and \cite{Mc+} for a
survey of the cases $n=2,3,4$. In \cite{g1inv} we gave a practical
algorithm for evaluating the invariants in the case $n=5$ and showed
that they give a formula for the Jacobian.  We now extend this
invariant theoretic approach to give a formula for the covering map.

We work throughout over a field $K$ of characteristic not dividing
$6n$, where in due course we take $n=5$. Except in the following
paragraph, and at the end of Section~\ref{sec:covmap},
we assume for simplicity that $K$ is algebraically closed.

To explain the motivation for our work, let $E$ be an elliptic curve
over a number field $K$. For any integer $n \ge 2$ the quotient group
$E(K)/n E(K)$ injects into the $n$-Selmer group $S^{(n)}(E/K)$, which
is finite and effectively computable.  In an explicit $n$-descent
calculation one represents each element of the $n$-Selmer group by
(equations for) an elliptic normal curve $C \subset \PP^{n-1}$ with
Jacobian $E$. It is perhaps better to call $C$ a ``genus one normal
curve'' as it need not have any $K$-rational points.  The Selmer group
elements with $C(K) \not= \emptyset$ make up the image of $E(K)/n
E(K)$ in $S^{(n)}(E/K)$.  Moreover if $P \in C(K)$ then a coset
representative for the corresponding element of $E(K)/nE(K)$ is given
by the image of $P$ under the covering map. Having explicit formulae
for the covering map can therefore help in finding generators for the
Mordell-Weil group $E(K)$.

In the case $n=5$ the curves of Definition~\ref{defdegen} are called
{\em elliptic normal quintics} and {\em rational nodal quintics}.  By
the Buchsbaum-Eisenbud structure theorem~\cite{BE1}, \cite{BE2} they
are defined by the $4 \times 4$ Pfaffians of a $5 \times 5$
alternating matrix of linear forms on $\PP^4$. We call such a matrix
$\phi$ a {\em genus one model} and write $C_\phi \subset \PP^4$ for
the subvariety defined by the $4 \times 4$ Pfaffians.  It is shown in
\cite[Proposition~5.10]{g1inv} that $C_\phi$ is a smooth curve of
genus one if and only if it is an elliptic normal quintic. In this
case we say that $\phi$ is {\em non-singular}.

There is a natural action of $\GL_5 \times \GL_5$ on the space of
genus one models. The first factor acts as $M : \phi \mapsto M \phi
M^T$ and the second factor acts by changing co-ordinates on $\PP^4$.
We adopt the following notation. Let $V$ and $W$ be $5$-dimensional
vector spaces with bases $v_0, \ldots, v_4$ and $w_0, \ldots, w_4$.
We identify the space of genus one models with $\wedge^2 V \otimes W$
via
\[ \phi = (\phi_{ij}) \longleftrightarrow \textstyle\sum_{i<j} (v_i
\wedge v_j) \otimes \phi_{ij}(w_0, \ldots, w_4). \] With this
identification the action of $\GL_5 \times \GL_5$ becomes the natural
action of $\GL(V) \times \GL(W)$ on $\wedge^2 V \otimes W$.  By
squaring and then identifying $\wedge^4 V \isom V^*$ there is a
natural map
\begin{equation}
\label{def:P2}
 P_2 : \wedge^2 V \otimes W \to V^* \otimes S^2 W = \Hom(V,S^2W). 
\end{equation}
Explicitly $P_2(\phi) = (v_i \mapsto p_i(w_0,\ldots, w_4))$ where
$p_0, \ldots, p_4$ are the $4 \times 4$ Pfaffians of $\phi$.  Thus $V$
may be thought of as the space of quadrics defining $C_{\phi}$ and $W$
as the space of linear forms on $\PP^4$.

\begin{lem} 
\label{lem:orbits}
The action of $\GL(V) \times \GL(W)$ is transitive on the the genus
one models $\phi$ for which $C_\phi$ is a rational nodal quintic, and
on the genus one models $\phi$ for which $C_\phi$ is an elliptic
normal quintic with given $j$-invariant.
\end{lem}
\begin{proof}
See \cite[Proposition 4.6]{g1inv}.
\end{proof}

The co-ordinate ring $K[\wedge^2 V \otimes W]$ is a polynomial ring in
$50$ variables.

\begin{thm}
\label{thm:inv} 
The ring of invariants for $\SL(V) \times \SL(W)$ acting on
$K[\wedge^2 V \otimes W]$ is generated by invariants $c_4$ and $c_6$
of degrees $20$ and $30$. Moreover if we scale them as specified in
\cite{g1inv} and put $\Delta = (c_4^3 - c_6^2)/1728$ then
\begin{enumerate}
\item a genus one model $\phi$ is non-singular if and only if
  $\Delta(\phi) \not=0$,
\item if $\phi$ is non-singular then $C_\phi$ has $j$-invariant
  $c_4(\phi)^3/\Delta(\phi)$.
\end{enumerate}
\end{thm}
\begin{proof}
See \cite[Theorem 4.4]{g1inv}.
\end{proof}

\begin{lem}
\label{lem:codim1}
Let $\phi \in \wedge^2 V \otimes W$ be a genus one model with $C_\phi$
either an elliptic normal quintic or a rational nodal quintic. Then
the Zariski closure of the $\GL(V) \times \GL(W)$-orbit of $\phi$ is
the zero locus of an irreducible homogeneous invariant $I$. Moreover
we can take
\[ I = \left\{ \begin{array}{cl} c_4 & \mbox{ if } j(C_\phi) = 0 \\
c_6 & \mbox{ if } j(C_\phi) = 1728  \\
\Delta & \mbox{ if $C_\phi$ is a rational nodal quintic} \\
c_4^3 - j(C_\phi) \Delta & \mbox{ otherwise. } \end{array} \right. \]
\end{lem}
\begin{proof}
  The existence of $I$ is proved in \cite[Lemma 4.10]{g1inv}.  The
  invariants listed vanish at $\phi$ by Theorem~\ref{thm:inv} and are
  irreducible in $K[c_4,c_6]$. They are therefore irreducible in
  $K[\wedge^2 V \otimes W]$ since any factors would themselves have to
  be invariants. We use here that $\SL(V) \times \SL(W)$ is connected
  and has no $1$-dimensional rational representations. Alternatively
  we can prove irreducibility by restricting to the Weierstrass models
  in \cite[Section 6]{g1inv}.
\end{proof}

\begin{lem}
\label{lem:van}
Let $I$ be a non-constant homogeneous invariant.  Then there exists
$\phi \in \wedge^2 V \otimes W$ with $I(\phi) = 0$ and $C_\phi$ either
an elliptic normal quintic or a rational nodal quintic.
\end{lem}
\begin{proof}
  We may assume that $I$ is irreducible in $K[c_4,c_6]$. So up to
  scalar multiples we have $I = c_4, c_6, \Delta$ or $c_4^3 - j
  \Delta$ with $j \not= 0,1728$. We take $C_\phi$ to be an elliptic
  normal quintic with the appropriate $j$-invariant, or in the case $I
  = \Delta$ a rational nodal quintic.
\end{proof}

The covariants we need to describe the covering map are $\SL(V) \times
\SL(W)$-equivariant polynomial maps $\wedge^2 V \otimes W \to S^{5d}
W$ for $d=1,2,3$. More generally we defined a {\em covariant} to be an
$\SL(V) \times \SL(W)$-equivariant polynomial map $\wedge^2 V \otimes
W \to Y$ where $Y$ is a rational representation of $\GL(V) \times
\GL(W)$. In all our examples $Y$ will be {\em homogeneous} by which we
mean there exist integers $r$ and $s$ such that the morphism $\rho_Y :
\GL(V) \times \GL(W) \to \GL(Y)$ satisfies $\rho_Y(\lambda I_V, \mu
I_W)= \lambda^r \mu^s I_Y$ for all $\lambda, \mu \in K^\times$.

\begin{lem}
\label{wtlemma}
Let $Y$ be a homogeneous rational representation of $\GL(V) \times
\GL(W)$ with degrees $(r,s)$. If $F: \wedge^2 V \otimes W \to Y$ is a
homogeneous covariant then there exist integers $p$ and $q$ called the
{\em weights} of $F$ such that
\begin{equation}
\label{wtformula}
\begin{aligned}
2 \deg F & = 5p+r \\ 
\deg F & = 5q+s.
\end{aligned}
\end{equation}
\end{lem}
\begin{proof}
See \cite[Lemma 2.2]{invenqI}. 
\end{proof}

For example the Pfaffian map~(\ref{def:P2}) is a covariant of degree
$2$ with weights $(p,q) = (1,0)$.  The covariants in the case $Y$ is
the trivial representation are the invariants as described in
Theorem~\ref{thm:inv}.  For general $Y$ the covariants form a module
over the ring of invariants $K[c_4,c_6]$.

In Section~\ref{sec:discov} we recall our method~\cite{invenqI} for
studying the covariants via their restrictions to the Hesse family,
i.e. the universal family over $X(5)$.  These restrictions are nearly
characterised by their invariance properties under an appropriate
action of $\SL_2(\Z/5\Z)$. In Sections~\ref{sec:frac}
and~\ref{sec:denom} we make this relationship precise. Thus our work
resolves, albeit in one particular case, what is described in
\cite[Chapter V,\S22]{AR} as the ``mysterious role of invariant
theory''.  We give examples for a range of different $Y$ in
Section~\ref{sec:ex}. In Section~\ref{sec:indep} we show how a free
basis for the $K[c_4,c_6]$-module of covariants for $Y$ may be
characterised in terms of its specialisations to the genus one models
$\phi$ of the form considered in Lemma~\ref{lem:codim1}.  In
Section~\ref{sec:quintics} we relate the covariants in the case
$Y=S^{5} W$ to work of Hulek \cite{Hu} and finally in
Section~\ref{sec:covmap} we give our formula for the covering map.

\section{Discrete covariants}
\label{sec:discov}

In this section we recall some of the theory from \cite{invenqI}. We
then state our main result on the relationship between covariants and
discrete covariants.

We take $n \ge 5$ an odd integer.  The {\em Heisenberg group} of level
$n$ is
\[ H_n = \langle \sigma, \tau | \sigma^n = \tau^n = [ \sigma,
[\sigma,\tau]] = [\tau, [\sigma,\tau]] = 1 \rangle. \] It is a
non-abelian group of order $n^3$ and its centre is a cyclic group of
order $n$ generated by $\zeta=[\sigma,\tau] = \sigma \tau \sigma^{-1}
\tau^{-1}$.  In \cite[Section 3]{invenqI} we defined a group
homomorphism $s_\beta : \GL_2(\Z/n\Z) \to \Aut(H_n)$ by \[ s_\beta ((
\begin{smallmatrix} a & b \\ c & d \end{smallmatrix} )) 
: \sigma \mapsto \zeta^{-ac/2} \sigma^a \tau^c \, ;
\, \, \tau \mapsto \zeta^{-bd/2} \sigma^b \tau^d.\]
where the exponents are read as integers mod $n$.

\begin{defn}
\label{defexheisen}
The {\em extended Heisenberg group} is the semi-direct product
\[ H_n^+ = H_n \ltimes \SL_2(\Z/n\Z), \]
with group law $ (h,\gamma) (h',\gamma')= (h \, s_\beta(\gamma)h', 
\gamma \gamma')$.
\end{defn}

The {\em Schr\"odinger representation} $\theta : H_n \to \SL_n(K)$
maps $\sigma$ and $\tau$ to the matrices~(\ref{heismats}).  These
matrices have commutator $\theta(\zeta) = \zeta_n I_n$.

\begin{thm}
\label{thm:exheis}
(i) The Schr\"odinger representation $\theta:H_n \to \SL_n(K)$ extends
uniquely to a representation $\theta^+:H_n^+ \to \SL_n(K)$. \\
(ii) The normaliser of $\theta(H_n)$ in $\SL_n(K)$ is
$\theta^+(H_n^+)$.
\end{thm}
\begin{proof} See \cite[Theorem 3.6]{invenqI}.
\end{proof}

\begin{rem}
  (i) The representation $\theta^+$ of Theorem~\ref{thm:exheis} is
  given on the generators $S = \left( \begin{smallmatrix} 0 & 1 \\ - 1
      & 0 \end{smallmatrix} \right)$ and $T = \left(
    \begin{smallmatrix} 1 & 1 \\ 0 & 1 \end{smallmatrix} \right)$ for
  $\SL_2(\Z/n\Z)$ by suitable scalar multiples of
\[ 
 \left( \begin{array}{c@{\quad}c@{\quad}c@{\quad}c@{\quad}c}
    1   &   1  &  1  & \cdots &      1     \\
   1   & \zeta_n  &  \zeta_n^2   & \cdots &      \zeta_n^{-1}   \\
    1   & \zeta_n^2   & \zeta_n^4 & \cdots &      \zeta_n^{-2}  \\ 
 \vdots & \vdots & \vdots  &        &   \vdots   \\
    1   &    \zeta_n^{-1} \!\!\! &  \zeta_n^{-2} \!\!\! & \cdots & \zeta_n  
\end{array} \right)  \quad \text{ and } \quad 
\left( \begin{array}{c@{\quad}c@{\,\,\,}c@{\,\,\,}c@{\,\,\,}c}
    1   &    0   &    0    & \cdots &      0     \\
    0   & \zeta^{1/2}_n \!\!\! &    0    & \cdots &      0     \\
    0   &    0   & \zeta_n^{2^2/2} \!\!\! & \cdots &      0     \\ 
 \vdots & \vdots & \vdots  &        &   \vdots   \\
    0   &   0    &    0    & \cdots & \zeta_n^{1/2}  
\end{array} \right). \]
(ii) The Schr\"odinger representation has $\phi(n)$ conjugates
obtained by either changing our choice of $\zeta_n$ or precomposing
with an automorphism of $H_n$. We may apply Theorem~\ref{thm:exheis}
to any one of these representations.
\end{rem}

The Hesse family of elliptic normal quintics (studied for example in
\cite{g1hess}, \cite{Hu}) is given by
\begin{equation}
\label{uab}
\begin{array}{lccl}
u:& \Aff^2 & \to & \wedge^2 V \otimes W \\
& (a,b) & \mapsto & a \sum (v_{1} \wedge v_{4}) w_0 +
b  \sum (v_{2} \wedge v_{3}) w_0
\end{array} 
\end{equation}
where the sums are taken over all cyclic permutations of the
subscripts mod $5$.  We define actions of the Heisenberg group $H_5$
on $V$ and $W$ so that the Hesse models $u(a,b)$ are $H_5$-invariant.
\begin{equation}
\label{thetavw}
\begin{array}{lll} \smallskip
\theta_V: H_5 \to \SL(V)\,; & \sigma : v_i \mapsto \zeta_5^{2i} v_i\,;
& \tau : v_i \mapsto v_{i+1} \\
\theta_W: H_5 \to \SL(W)\,; & \sigma : w_i \mapsto \zeta_5^{i} w_i\,;
&  \tau : w_i \mapsto w_{i+1}.
\end{array} 
\end{equation}
Since $\theta_V$ and $\theta_W$ are conjugates of the Schr\"odinger
representation they extend by Theorem~\ref{thm:exheis} to
representations of $H_5^+$. By abuse of notation we continue to write
these representations as $\theta_V$ and $\theta_W$.

Let $Y$ be a homogeneous rational representation of $\GL(V) \times
\GL(W)$. Then $\theta_V$ and $\theta_W$ define an action of $H_5^+$ on
$Y$ and so an action of $\Gamma = \SL_5(\Z/5\Z)$ on $Y^{H_5}$. Taking
$Y = \wedge^2 V \otimes W$ the action of $\Gamma$ on $(\wedge^2 V
\otimes W)^{H_5} = {\operatorname{Im}}(u)$ is described by a
representation $\chi_1 : \Gamma \to \GL_2(K)$.

\begin{defn}
  Let $\pi : \Gamma \to \GL(Z)$ be a representation. A {\em discrete
    covariant} for $Z$ is a polynomial map $f : \Aff^2 \to Z$
  satisfying $f \circ \chi_1(\gamma) = \pi (\gamma) \circ f$ for all
  $\gamma \in \Gamma$.
\end{defn} 

\begin{thm} 
\label{thm:dcov}
Let $F : \wedge^2 V \otimes W \to Y$ be a covariant. Then $f = F \circ
u : \Aff^2 \to Y^{H_5}$ is a discrete covariant. Moreover $F$ is
uniquely determined by $f$.
\end{thm}
\begin{proof}
See \cite[Theorem 4.3]{invenqI}.
\end{proof}

For any given $Y$ the discrete covariants may be computed using
invariant theory for the finite groups $H_5$ and $\SL_2(\Z/5\Z)$.  We
say that a discrete covariant $f : \Aff^2 \to Y^{H_5}$ {\em is a
  covariant} if it arises from a covariant $F: \wedge^2 V \otimes W
\to Y$ as described in Theorem~\ref{thm:dcov}. It is important to note
that not every discrete covariant is a covariant.  For example, taking
$Y$ to be the trivial representation, the ring of invariants is
$K[c_4,c_6]$ as described in Theorem~\ref{thm:inv} whereas the ring of
discrete invariants is generated by
\begin{equation}
\label{discinv}
\begin{aligned}
D & =  ab(a^{10}-11a^5 b^5-b^{10}) \\ 
c_4 & = 
  a^{20} + 228 a^{15} b^5 + 494 a^{10} b^{10} - 228 a^5 b^{15} + b^{20} \\
c_6 & = 
 -a^{30} + 522 a^{25} b^5 + 10005 a^{20} b^{10} 
+ 10005 a^{10} b^{20} - 522a^5 b^{25} - b^{30}
\end{aligned}
\end{equation}
subject only to the relation $c_4^3-c_6^2 = 1728 D^5$.  We use the
same notation for both a covariant and its restriction to the Hesse
family. By the uniqueness part of Theorem~\ref{thm:dcov} this should
not cause any confusion.

There are essentially two ways in which a discrete covariant might
fail to be a covariant. The first is that the weights computed
using~(\ref{wtformula}) might not be integers.  For example $D$ has
weights $(p,q)=(24/5,12/5)$ and so cannot be an invariant. The second
is that denominators might be introduced.  More precisely we prove the
following theorem in Section~\ref{sec:frac}.

\begin{thm}
\label{mainthm}
Let $f:\Aff^2 \to Y^{H_5}$ be an integer weight discrete covariant.
Then $\Delta^k f$ is a covariant for some $k \ge 0$.
\end{thm}

In Section~\ref{sec:denom} we give a practical method for computing
the least such $k$.

\begin{rem}
  If $Y$ is homogeneous of degree $(r,s)$ and $Y^{H_5} \not= 0$ then
  the action of the centre of $H_5$ shows that $2r+s \equiv 0
  \pmod{5}$. We see by~(\ref{wtformula}) that $p$ is an integer if and
  only if $q$ is an integer.  So the integer weight condition is just
  a congruence mod~$5$ on the degree of a covariant. Since $\Delta =
  D^5$ and $\deg D =12$ is coprime to $5$, an equivalent formulation
  of Theorem~\ref{mainthm} is that if $f:\Aff^2 \to Y^{H_5}$ is a
  homogeneous discrete covariant then $D^m f$ is a covariant for some
  $m \ge 0$.
\end{rem}

\section{Fractional covariants}
\label{sec:frac}

In this section we prove Theorem~\ref{mainthm}.

\begin{lem}
\label{handn}
Let $\phi \in \wedge^2 V \otimes W$ be a non-singular Hesse model.
\begin{enumerate}
\item The stabiliser of $\phi$ in $\SL(V) \times \SL(W)$ is
\[ H =\{  (\theta_V (h), \theta_W(h)) \, : \, h \in H_5 \}. \] 
\item The normaliser of $H$ in $\SL(V) \times \SL(W)$ is
\[ N =\{ (\theta_V (h), \zeta \theta_W(h)) 
\, : \, (\zeta,h)  \in \mu_5 \times H_5^+  \}. \]
\end{enumerate}
\end{lem}

\begin{proof}
  (i) It is clear by (\ref{uab}) and (\ref{thetavw}) that $H$ is
  contained in the stabiliser of~$\phi$. Since any automorphism of
  $C_\phi$ of order $5$ is translation by a $5$-torsion point of its
  Jacobian, all such automorphisms are described by elements of $H$.

  Now let $g \in \SL(V) \times \SL(W)$ with $g (\phi) = \phi$ and let
  $\gamma$ be the automorphism of $C_\phi$ induced by $g$. By
  \cite[Proposition 5.19 and Lemma 2.4]{g1inv} $\gamma$ preserves the
  invariant differential 
  and is therefore a translation map. Since $C_\phi \subset \PP^4$ is
  a curve of degree~$5$ this translation is by a point of order~$5$.
  Composing $g$ with a suitable element of $H$ reduces us to the case
  $\gamma$ is the identity. Then $g = (g_V,g_W)$ is a pair of scalar
  matrices.  Since these matrices each have determinant~$1$ and
  jointly fix $\phi$ it follows that $(g_V,g_W)=
  (\theta_V(h),\theta_W(h))$ for some $h$ in the centre of $H_5$.

  (ii) We see by Theorem~\ref{thm:exheis}(ii) that $N$ is contained in
  the normaliser of $H$, and that any element of the normaliser may be
  composed with an element of $N$ to give an element of the form $g =
  (I_V, g_W)$ where $I_V$ is the identity.  Since $\theta_V$ is
  faithful it follows that $g_W$ is in the centraliser of
  $\theta_W(H_5)$ in $\SL(W)$, which turns out to consist only of
  scalar matrices.
\end{proof}

The following proposition will be used to explain the relationship
between the covariants and the discrete covariants.

\begin{prop}
\label{prop1}
Let $G$ be a linear algebraic group acting on irreducible affine
varieties $X$ and $Y$. Let $H \subset G$ be a subgroup whose
normaliser $N \subset G$ is of finite order coprime to $\charic K$.
Suppose that $A \subset X^H$ is an irreducible variety acted on by
$N/H$, and $U \subset A$ is a dense open subset such that
{\renewcommand{\theenumi}{\roman{enumi}}
\begin{enumerate}
\item the morphism $G \times U \to X; \, (g,\phi) \mapsto g(\phi)$ has
  dense image,
\item the stabiliser in $G$ of each element of $U$ is $H$,
\item either $\charic K =0$ or the derivative of the map in (i) is an
  isomorphism at all points of $G \times U$.
\end{enumerate}}
\noindent
Then by restriction to $A$ there is a bijection between
\begin{itemize}
\item $G$-equivariant rational maps $F : X \ratto Y$, and
\item $N/H$-equivariant rational maps $f : A \ratto Y^H$
\end{itemize}
\end{prop}

\begin{proof}
  Let $F : X \ratto Y$ be a $G$-equivariant rational map. Its domain
  of definition is a $G$-invariant open subset of $X$ and hence by (i)
  it meets $U$. Therefore $F$ restricts to a rational map $f$ on $A$.
  By hypothesis $A$ is acted on by $N$ and pointwise fixed by $H$.
  Since $F$ is $N$-equivariant it follows that $f(A) \subset Y^H$ and
  $f$ is $N/H$-equivariant.

  Conversely suppose $f : A \ratto Y^H$ is an $N/H$-equivariant
  rational map. We let $\delta \in N$ act on $G \times A$ via $(g,a)
  \mapsto (g\delta^{-1}, \delta a)$. Since $N$ is a finite group of
  order coprime to $\charic K$ and $G \times A$ is an affine variety,
  the quotient $(G \times A)/N$ exists, and is an affine variety.  We
  consider the maps
\begin{align*}
\psi_{\id} &:  (G \times A)/N \ra X; \hspace{-5em} 
& (g,a) & \mapsto g(a) \\
\psi_{f} &:  (G \times A)/N \ratto Y; \hspace{-5em}
 & (g,a) & \mapsto g(f(a)).
\end{align*}
Shrinking $U$ if necessary, we may assume that $N/H$ acts on $U$.  By
(i) $\psi_{\id}$ has dense image, by (ii) it is injective on the dense
subset $(G \times U)/N$, and by (iii) it is separable.  It
follows 
that $\psi_{\id}$ is birational. Then $F = \psi_f \circ
\psi_{\id}^{-1}$ is a $G$-equivariant rational map extending $f$.
\end{proof}

\begin{ProofOf}{Theorem~\ref{mainthm}}
  We apply Proposition~\ref{prop1} with $G = \SL(V) \times \SL(W)$, $X
  = \wedge^2 V \otimes W$ and $H \subset N \subset G$ as in
  Lemma~\ref{handn}.  We also let $A = X^H$ be the space of Hesse
  models and $U \subset A$ the space of non-singular Hesse models.

  We check the hypotheses (i), (ii) and (iii).  By \cite[Proposition
  4.1]{g1hess} every non-singular model is equivalent to a Hesse model
  and by Theorem~\ref{thm:inv} the non-singular models are Zariski
  dense in $\wedge^2 V \otimes W$. This proves (i). We checked (ii) in
  Lemma~\ref{handn} and (iii) is checked in Lemma~\ref{lem:der} below.

  By Lemma~\ref{handn} and the definition of $H_5^+$ we have $N/H
  \isom \mu_5 \times \Gamma$ where $\Gamma = \SL_2(\Z/5\Z)$.  Now $f$
  is $\Gamma$-equivariant by definition of a discrete covariant and
  $\mu_5$-equivariant by the assumption it has integer weights.  So by
  Proposition~\ref{prop1} it is the restriction of a $G$-equivariant
  rational map $F : \wedge^2 V \otimes W \ratto Y$.  (We say $F$ is a
  {\em fractional covariant}.)

  It remains to show that $\Delta^k F$ is regular for some $k \ge 0$.
  Let $S \in K[\wedge^2 V \otimes W]$ be a homogeneous polynomial of
  least degree such that $S F$ is regular. Then $F = R/S$ where $R$ is
  a covariant and $S$ is an invariant. Suppose $S(\phi)=0$ for some
  non-singular model $\phi$.  By \cite[Proposition 4.1]{g1hess} we may
  suppose that $\phi$ is a Hesse model, and so by the regularity of
  $f$ we have $R(\phi) = S(\phi)=0$. By Lemma~\ref{lem:codim1} the
  Zariski closure of the $\GL(V) \times \GL(W)$-orbit of $\phi$ is the
  zero locus of a homogeneous invariant $I$. Now both $R$ and $S$ are
  divisible by $I$ and this contradicts the choice of $S$. Therefore
  $F$ is regular on all non-singular models.  By
  Theorem~\ref{thm:inv}(i) and the Nullstellensatz it follows that
  $\Delta^k F$ is regular for some $k \ge 0$.
\end{ProofOf}

The following lemma completes the proof of Theorem~\ref{mainthm} in
the case of positive characteristic (still assuming $\charic K
\not=2,3,5$).

\begin{lem}
\label{lem:der}
The derivative of the morphism  
\[ \begin{array}{rl}
\SL(V) \times \SL(W) \times \Aff^2 & \to \wedge^2 V \otimes W \\
(g_V,g_W,(a,b)) & \mapsto (g_V,g_W) u(a,b) 
\end{array} \] 
is an isomorphism at all $(g_V,g_W,(a,b))$ with $D(a,b) \not=0$.
\end{lem}
\begin{proof}
It suffices to compute the derivative at $(I_V,I_W,(a,b))$.  This is
a linear map
\begin{equation}
\label{derivative}
\liesl(V) \times \liesl(W) \times \Aff^2 \to \wedge^2 V \otimes W. 
\end{equation}
We write $E_{ij}$ for the $n \times n$ matrix with $(i,j)$ entry 1 and
all other entries 0. Then $\liesl_n$ has basis $ \{ E_{ij} : i \not= j
\} \cup \{ E_{00}-E_{ii} : i \not = 0 \}.$ Taking these bases for
$\liesl(V)$ and $\liesl(W)$, the standard basis for $\Aff^2$ and the
basis $\{(v_i \wedge v_j) w_k : i <j \}$ for $\wedge^2 V \otimes W$,
we found by direct calculation that the derivative~(\ref{derivative})
has determinant $5^4 D(a,b)^4$.
\end{proof}

\section{Denominators}
\label{sec:denom}

In this section we show how to find the least value of $k$ in
Theorem~\ref{mainthm}.  We consider the family of genus one models
\begin{equation}
\label{u1def}
\begin{array}{rl}
u_1 : \Aff^5 & \to \wedge^2 V \otimes W \\ 
(\la_0, \ldots , \la_4) & \mapsto 
  \textstyle\sum \lambda_0 (v_{1} \wedge v_{4}) w_0 +
\sum (v_{2} \wedge v_{3}) w_0.
\end{array} 
\end{equation}
where the sums are taken over all cyclic permutations of the
subscripts mod $5$. These models are related to the Hesse family by
\begin{equation}
\label{x1x}
u_1(a,\ldots,a) = u(a,1).
\end{equation}

\begin{rem} 
  If $\phi=u_1(\la_0, \ldots , \la_4)$ then $C_\phi \subset \PP^4$ is
  defined by $\lambda_0 x_0^2 + x_1 x_4 - \lambda_2 \lambda_3 x_2 x_3
  = 0$ and its cyclic permutes. These curves were studied in
  \cite{jems} where it is shown that $\phi = u_1(\la,1,\ldots,1)$
  defines the universal family of (generalised) elliptic curves
  parametrised by $X_1(5)$. Here $\la$ is a co-ordinate on $X_1(5)
  \isom \PP^1$.
\end{rem}

\begin{defn}
Let $D \subset \SL(V) \times \SL(W)$ be the subgroup of pairs
of diagonal matrices 
\begin{equation}
\label{diagvw}
\left( \begin{array}{ccccc}\alpha_1 \alpha_4 \!\! \\ 
&       \!\! \alpha_0 \alpha_2 \!\! \\ 
& &     \!\! \alpha_1 \alpha_3 \!\! \\
& & &   \!\! \alpha_2 \alpha_4 \!\! \\ 
& & & & \!\! \alpha_0 \alpha_3 
\end{array} \right), \quad 
\left( \begin{array}{ccccc} \alpha_0 \, \\  
&       \, \alpha_1 \, \\
& &     \, \alpha_2 \, \\
& & &   \, \alpha_3 \, \\
& & & & \, \alpha_4 
\end{array} \right),
\end{equation} 
with $\prod \alpha_i = 1$.
\end{defn}

\begin{lem} 
\label{lem:D}
The action of $D$ on $\Aff^5$ compatible with $u_1$ is
\[ (\la_0, \la_1, \la_2, \la_3,\la_4) \mapsto 
(\textstyle\frac{\alpha_0^2}{\alpha_1 \alpha_4} \la_0, 
\frac{\alpha_1^2}{\alpha_0 \alpha_2} \la_1,
\frac{\alpha_2^2}{\alpha_1 \alpha_3} \la_2, 
\frac{\alpha_3^2}{\alpha_2 \alpha_4} \la_3,
\frac{\alpha_4^2}{\alpha_0 \alpha_3} \la_4). \]
In particular $D$ acts transitively on the subsets of $\Aff^5$ defined 
by the condition that $\la_0, \ldots, \la_4$ have a fixed non-zero product.
\end{lem}
\begin{proof}
  Let $g_V$ and $g_W$ be the matrices~(\ref{diagvw}) with $(\alpha_0,
  \ldots, \alpha_4) = (\alpha,1, \ldots,1)$. Then
\[ (g_V,g_W) \, u_1(\la_0,\ldots,\la_4) = \alpha \, 
u_1(\alpha^2 \la_0,\alpha^{-1} \la_1, \la_2,\la_3,\alpha^{-1}\la_4). \]
From this calculation and the obvious cyclic symmetry it follows
that the action of $D$ on $\Aff^5$ is as stated. In the special
case $(\alpha_0, \ldots,\alpha_4) = (\beta^{-2},\beta^{-1},1,\beta,\beta^2)$ 
this action is given by 
$(\la_0, \la_1, \la_2, \la_3,\la_4) \mapsto 
(\beta^{-5} \la_0, \la_1, \la_2,  \la_3, \beta^5 \la_4)$.
Since we are working over an algebraically closed field 
the final statement is clear.
\end{proof}

Let $Y$ be a homogeneous rational representation of $\GL(V) \times \GL(W)$.

\begin{thm}
\label{f1thm}
Let 
$f : \Aff^2 \to Y^{H_5}$ be an integer weight discrete covariant.
\begin{enumerate}
\item There is a unique $D$-equivariant rational map $f_1 : \Aff^5
  \ratto Y$ with \[f_1(a,\ldots,a) = f(a,1).\]
\item $f$ is a covariant if and only if $f_1$ is regular.
\end{enumerate}
\end{thm}

\begin{proof}
  By Theorem~\ref{mainthm} there is a fractional covariant $F :
  \wedge^2 V \otimes W \ratto Y$ with $f = F \circ u$. It follows
  by~(\ref{x1x}) and Lemma~\ref{lem:D} that $f_1 = F \circ u_1$
  satisfies (i).  Uniqueness is proved using the final part of
  Lemma~\ref{lem:D}.

  It remains to show that if $f_1$ is regular then $F$ is regular.
  Theorem~\ref{mainthm} already shows that $R = \Delta^k F$ is a
  covariant for some $k \ge 0$.  We take the least such $k$. Let $\phi
  = u_1(0,1,1,1,1)$. Then $C_\phi$ is the rational nodal quintic
  parametrised by
  \[ (x_0: \ldots :x_4) = (s^5-t^5:st^4:s^2t^3:-s^3t^2:-s^4t). \] If
  $k \ge 1$ then by regularity of $f_1$ we have $R(\phi) = 0$.  Then
  Lemma~\ref{lem:codim1} shows that $R$ is divisible by $\Delta$
  contradicting our choice of $k$. Therefore $k=0$ and $F$ is a
  covariant. By the convention introduced following
  Theorem~\ref{thm:dcov}, we say that $f$ is a covariant.
\end{proof}

What makes Theorem~\ref{f1thm} useful is that we can compute $f_1$
from $f$ without going via $F$. Explicitly we put
\begin{equation}
\label{convert}
f_1(\la_0, \ldots, \la_4) = \rho_Y(g_V,g_W) f(a,1) 
\end{equation}
where $g_V$ and $g_W$ are given by~(\ref{diagvw}) and satisfy
$u_1(\la_0, \ldots, \la_4) = (g_V,g_W) u(a,1)$.  We then eliminate
$\alpha_0, \ldots, \alpha_4$ and $a$ from the right hand side, using
the relations
\begin{eqnarray}
\label{subst1}
\alpha_i^2 /(\alpha_{i+1} \alpha_{i+4}) & = & \la_i /a \\
\label{subst2}
\alpha_i^5 & = & \la_i^2 /(\la_{i+2} \la_{i+3}) \\
\label{subst3} 
a^5 & = & \la_0 \la_1 \ldots \la_4. 
\end{eqnarray}
The first of these comes from Lemma~\ref{lem:D}. The other two may be
deduced from the first using $\prod \alpha_i=1$.  One systematic way
to proceed is by using~(\ref{subst1}) to eliminate $\alpha_0$,
$\alpha_1$, $\alpha_2$, then~(\ref{subst2}) to eliminate $\alpha_3$,
$\alpha_4$ and finally~(\ref{subst3}) to eliminate $a$.

\begin{rem}
\label{rem:T}
It can be shown that Theorem~\ref{f1thm}(i) still holds if we weaken
the condition that $f$ is $\SL_2(\Z/5\Z)$-equivariant and just require
that it is equivariant for the action of $T = (\begin{smallmatrix} 1 &
  1 \\ 0 & 1 \end{smallmatrix})$.
\end{rem} 

\section{Examples}
\label{sec:ex}

We can use Theorem~\ref{f1thm} in the case $Y$ is the trivial
representation to give another proof (independent of
Theorem~\ref{thm:inv}) that the discrete invariants $c_4$ and $c_6$
are in fact invariants.  Indeed let $f$ be an integer weight discrete
invariant. The integer weight condition is that $f$ is homogeneous of
degree a multiple of $5$.  We construct $f_1$ from $f$ by making the
substitutions $a^5 \mapsto \prod \la_i$ and $b \mapsto 1$. Since no
denominators are introduced it follows by Theorem~\ref{f1thm} that $f$
is an invariant.

In the cases $Y = \wedge^2 V \otimes W$ and $\wedge^2 V^* \otimes W^*$
the following proposition was already proved in \cite{g1hess} using
evectants. We now have a general method. In the calculations that
follow all sums and products are taken over the cyclic permutations of
the subscripts mod $5$. Recall that we fixed bases $v_0, \ldots, v_4$
and $w_0, \ldots, w_4$ for $V$ and $W$. The dual bases for $V^*$ and
$W^*$ are $v^*_0, \ldots, v^*_4$ and $w^*_0, \ldots, w^*_4$.

\begin{prop}
Let $Y$ be any one of 
\[ \begin{array}{cccc}
\wedge^2 V \otimes W,     &
V^* \otimes \wedge^2 W,   &
V \otimes \wedge^2 W^*,   & 
\wedge^2 V^* \otimes W^*, \\
~\,\,\,\,\,\, V^* \otimes S^2W,        & 
\,\,\,\,\,\, S^2V^* \otimes W^*,       &
\,\,\,\,\,\, S^2V \otimes W,           &
\,\,\,\,\,\, V \otimes S^2W^*.         
\end{array} \]
Then every integer weight discrete covariant $f : \wedge^2 V \otimes W 
\to Y^{H_5}$ is a covariant. In particular the covariants
$F : \wedge^2 V \otimes W \to Y$ form a free $K[c_4,c_6]$-module of
rank $2$ or $3$ and the generators have degrees as indicated in 
\cite[Table 4.6]{invenqI}. 
\end{prop}
\begin{proof}
  Let $f : \wedge^2 V \otimes W \to Y^{H_5}$ be an integer weight
  discrete covariant. In each case \cite[Lemma 4.4]{invenqI} shows
  that $\dim Y^{H_5} = 2$ or $3$ and a basis is found by inspection.
  We construct $f_1$ from $f$ by making the substitutions $a^5 \mapsto
  \prod \la_i$, $b \mapsto 1$, and 
\small
\begin{align*} \\
a \textstyle\sum(v_1 \wedge v_4) w_0 
  & \mapsto \textstyle\sum \la_0 (v_1 \wedge v_4) w_0 &
  \textstyle\sum v_0^*(w_1 \wedge w_4) 
  & \mapsto \textstyle\sum v_0^*(w_1 \wedge w_4) \\
  \textstyle\sum(v_2 \wedge v_3) w_0 & 
  \mapsto \textstyle\sum(v_2 \wedge v_3) w_0 &  
a^2  \textstyle\sum v_0^*(w_2 \wedge w_3) & 
  \mapsto \textstyle\sum \la_2 \la_3 v_0^*(w_2 \wedge w_3) \\ \\
  \textstyle\sum v_0 (w^*_1 \wedge w^*_4)  
  & \mapsto \textstyle\sum v_0 (w^*_1 \wedge w^*_4) &
a^4 \textstyle\sum(v^*_1 \wedge v^*_4) w^*_0 
  & \mapsto \textstyle\sum \la_1 \la_2 \la_3 \la_4 
                                (v^*_1 \wedge v^*_4) w^*_0 \\
a^3 \textstyle\sum v_0 (w^*_2 \wedge w^*_3) & 
  \mapsto \textstyle\sum \la_0 \la_1 \la_4 v_0 (w^*_2 \wedge w^*_3) &  
  \textstyle\sum(v^*_2 \wedge v^*_3) w^*_0 & 
  \mapsto \textstyle\sum(v^*_2 \wedge v^*_3) w^*_0 \\ \\
a \textstyle\sum v_0^* w_0^2  & 
  \mapsto \textstyle\sum \la_0 v_0^* w_0^2 & 
a^2 \textstyle\sum v_0^{*2} w^*_0  & 
  \mapsto \textstyle\sum \la_2 \la_3 v_0^{*2} w^*_0 \\
  \textstyle\sum v_0^* w_1 w_4  & 
  \mapsto \textstyle\sum v_0^* w_1 w_4 &
a^4 \textstyle\sum v_1^* v_4^* w_0^*  & 
  \mapsto \textstyle\sum \la_1 \la_2 \la_3 \la_4 v_1^* v_4^* w_0^* \\
 a^2 \textstyle\sum v_0^* w_2 w_3  & 
  \mapsto \textstyle\sum \la_2 \la_3 v_0^* w_2 w_3 &
   \textstyle\sum v_2^* v_3^* w^*_0  & 
  \mapsto \textstyle\sum v_2^* v_3^* w^*_0 \\ \\ 
a^3 \textstyle\sum v_0^2 w_0  & 
  \mapsto \textstyle\sum \la_0 \la_1 \la_4 v_0^2 w_0 & 
a^4 \textstyle\sum v_0 w_0^{*2}  & 
  \mapsto \textstyle\sum \la_1 \la_2 \la_3 \la_4  v_0 w_0^{*2} \\
a \textstyle\sum v_1 v_4 w_0   & 
  \mapsto \textstyle\sum \la_0 v_1 v_4 w_0 &
  \textstyle\sum v_0 w_1^* w_4^*  & 
  \mapsto \textstyle\sum v_0 w^*_1 w^*_4 \\
    \textstyle\sum v_2 v_3 w_0  & 
  \mapsto \textstyle\sum v_2 v_3 w_0 &
 a^3 \textstyle\sum v_0 w^*_2 w^*_3  & 
  \mapsto \textstyle\sum \la_0 \la_1 \la_4 v_0 w^*_2 w^*_3 \\
\end{align*}
\normalsize 
Since these substitutions eliminate $a$ it is clear that
no denominators are introduced. It follows by Theorem~\ref{f1thm} that
$f$ is a covariant.
\end{proof}

\begin{prop}
\label{prop:quintics}
Let $Y$ be any one of $S^5 W, S^5 V, S^5 V^*, S^5 W^*$.  Then the
covariants $F : \wedge^2 V \otimes W \to Y$ form a free
$K[c_4,c_6]$-module of rank $6$ with generators in degrees $10, 20,
30, 30, 40, 50$ except in the case $Y = S^5 W^*$ where the generators
have degrees $30,40,50,50,60,70$.
\end{prop}

\begin{proof}
  The module of integer weight discrete covariants is computed as
  described in \cite{invenqI} and is found to be a free
  $K[c_4,c_6]$-module of rank $6$ with generators in degrees
  $10,10,20,20,30,30$. We use Theorem~\ref{f1thm} to decide which of
  these are covariants.  We construct $f_1$ from $f$ by making the
  substitutions $a^5 \mapsto \prod \la_i$, $b \mapsto 1$ and 
\small
\begin{align*} a^5 \textstyle\sum w_0^5 & \mapsto \textstyle\sum
    \la_0^3 \la_1 \la_4 w_0^5 & a^5 \textstyle\sum v_0^5 &
    \mapsto  \textstyle\sum \la_0 \la_1^2 \la_4^2 v_0^5 \\
    a^4 \textstyle\sum w_0^3 w_1 w_4 & \mapsto \textstyle\sum \la_0^2
    \la_1 \la_4 w_0^3 w_1 w_4 & a^3 \textstyle\sum v_0^3 v_1 v_4 &
    \mapsto  \textstyle\sum \la_0 \la_1 \la_4 v_0^3 v_1 v_4 \\
    a^3 \textstyle\sum w_0 w_1^2 w_4^2 & \mapsto \textstyle\sum \la_0
    \la_1 \la_4 w_0 w_1^2 w_4^2 & a \textstyle\sum v_0 v_1^2 v_4^2 &
    \mapsto  \textstyle\sum \la_0 v_0 v_1^2 v_4^2 \\
    a^2 \textstyle\sum w_0 w_2^2 w_3^2 & \mapsto \textstyle\sum \la_2
    \la_3 w_0 w_2^2 w_3^2 & a^4 \textstyle\sum v_0 v_2^2 v_3^2 &
    \mapsto  \textstyle\sum \la_1 \la_2 \la_3 \la_4 v_0 v_2^2 v_3^2 \\
    a \textstyle\sum w_0^3 w_2 w_3 & \mapsto \textstyle\sum \la_0
    w_0^3 w_2 w_3 & a^2 \textstyle\sum v_0^3 v_2 v_3 &
    \mapsto  \textstyle\sum \la_1 \la_4 v_0^3 v_2 v_3 \\
    \textstyle\prod w_0 & \mapsto \textstyle\prod w_0 &
    \textstyle\prod v_0 & \mapsto  \textstyle\prod v_0 \\ \\
    a^5 \textstyle\sum v^{*5}_0 & \mapsto \textstyle\sum \la_0 \la_2^2
    \la_3^2 v^{*5}_0 & a^{10} \textstyle\sum w^{*5}_0 &
    \mapsto  \textstyle\sum \la_1^2 \la_2^3 \la_3^3 \la_4^2 w^{*5}_0 \\
    a^2 \textstyle\sum v^{*3}_0 v^{*}_1 v^{*}_4 & \mapsto
    \textstyle\sum \la_2 \la_3 v^{*3}_0 v^{*}_1 v^{*}_4 & a^6
    \textstyle\sum w^{*3}_0 w^{*}_1 w^{*}_4 &
    \mapsto \textstyle\sum \la_1 \la_2^2 \la_3^2 \la_4 w^{*3}_0 w^{*}_1 w^{*}_4 \\
    a^4 \textstyle\sum v^{*}_0 v^{*2}_1 v^{*2}_4 & \mapsto
    \textstyle\sum \la_1 \la_2 \la_3 \la_4 v^{*}_0 v^{*2}_1 v^{*2}_4 &
    a^2 \textstyle\sum w^{*}_0 w^{*2}_1 w^{*2}_4 &
    \mapsto  \textstyle\sum \la_2 \la_3 w^{*}_0 w^{*2}_1 w^{*2}_4 \\
    a \textstyle\sum v^{*}_0 v^{*2}_2 v^{*2}_3 & \mapsto
    \textstyle\sum \la_0 v^{*}_0 v^{*2}_2 v^{*2}_3 & a^3
    \textstyle\sum w^{*}_0 w^{*2}_2 w^{*2}_3 &
    \mapsto  \textstyle\sum \la_0 \la_1 \la_4 w^{*}_0 w^{*2}_2 w^{*2}_3 \\
    a^3 \textstyle\sum v^{*3}_0 v^{*}_2 v^{*}_3 & \mapsto
    \textstyle\sum \la_0 \la_2 \la_3 v^{*3}_0 v^{*}_2 v^{*}_3 & a^4
    \textstyle\sum w^{*3}_0 w^{*}_2 w^{*}_3 &
    \mapsto  \textstyle\sum \la_1 \la_2 \la_3 \la_4 w^{*3}_0 w^{*}_2 w^{*}_3 \\
    \textstyle\prod v^{*}_0 & \mapsto \textstyle\prod v^{*}_0 &
    \textstyle\prod w^{*}_0 & \mapsto  \textstyle\prod w^{*}_0 \\
\end{align*} 
\normalsize

In the cases $Y = S^5 W, S^5 V, S^5 V^*$ an integer weight discrete
covariant is a covariant if and only if the coefficient of $\sum
w_0^5, \sum v_0^5, \sum v^{*5}_0$ is divisible by $a^5$.  Computing
the discrete covariants we find that there is a single constraint in
degree $10m$ for each $m \ge 1$.  The covariants therefore have
Hilbert series
\[ \frac{2(t^{10}+t^{20}+t^{30})}{(1-t^{20})(1-t^{30})} -
\frac{t^{10}}{1-t^{10}} =
\frac{t^{10}+t^{20}+2t^{30}+t^{40}+t^{50}}{(1-t^{20})(1-t^{30})}. \]
In Section~\ref{sec:quintics} we give further details of the
covariants in the case $Y = S^5 W$.

In the case $Y = S^5 W^*$ an integer weight discrete covariant is a
covariant if and only if the coefficient of $\sum w^{*5}_0$ is
divisible by $a^{10}$ and the coefficient of $\sum w_0^{*3} w_1^*
w_4^*$ is divisible by $a^{5}$.  We find that the discrete covariants
of degrees $10$ and $20$ are not covariants and that there are $3$
constraints in degrees $10m$ for each $m \ge 3$.  The covariants
therefore have Hilbert series
\[  \frac{2(t^{10}+t^{20}+t^{30})}{(1-t^{20})(1-t^{30})} 
- 2 t^{10} - 2 t^{20 }- \frac{3 t^{30}}{1-t^{10}} = 
\frac{t^{30}+t^{40}+2t^{50}+t^{60}+t^{70}}{(1-t^{20})(1-t^{30})}. \]
\end{proof}

\begin{example}
  The degree $10$ covariant for $Y = S^{5} W$ is 
\small
\begin{align*}
S_{10} & = a^5 b^5 \textstyle\sum w_0^5 - a^4 b (a^5-3b^5) 
 \textstyle\sum w_0^3 w_1 w_4 + a^3 b^2 (a^5+2 b^5) 
 \textstyle\sum w_0 w_1^2 w_4^2  \\  + & a^2 b^3 (2a^5-b^5) 
 \textstyle\sum w_0 w_2^2 w_3^2 - a b^4 (3 a^5+b^5) 
 \textstyle\sum w_0^3 w_2 w_3 + (a^{10}-16 a^5 b^5-b^{10}) \textstyle\prod w_0 
\end{align*}  
\normalsize
and the degree $30$ covariant for $Y = S^{5} W^*$ is
\small
\begin{align*}
T_{30} & = 125 a^{10} b^{10} (3 a^{10}
 - 8 a^5 b^5 - 3 b^{10})\textstyle\sum w^{*5}_0 \\
 - & 5 a^6 b^4 (3 a^{20} + 134 a^{15} b^5 + 57 a^{10} b^{10} 
+ 216 a^5 b^{15} - 22 b^{20})
\textstyle\sum w^{*3}_0 w^{*}_1 w^{*}_4 \\
 + & a^2 b^3 (32 a^{25} - 195 a^{20} b^5 + 4110 a^{15} b^{10} 
+ 900 a^{10} b^{15} + 480 a^5 b^{20} + 9 b^{25})
\textstyle\sum w^{*}_0 w^{*2}_1 w^{*2}_4 \\
 - & a^3 b^2 (9 a^{25} - 480 a^{20} b^5 + 900 a^{15} b^{10} 
- 4110 a^{10} b^{15} - 195 a^5 b^{20} - 32 b^{25})
\textstyle\sum w^{*}_0 w^{*2}_2 w^{*2}_3 \\
 - & 5 a^4 b^6 (22 a^{20} + 216 a^{15} b^5 - 57 a^{10} b^{10} 
+ 134 a^5 b^{15} - 3 b^{20}) 
\textstyle\sum w^{*3}_0 w^{*}_2 w^{*}_3 \\
 + & (a^{30} - 258 a^{25} b^5 + 3435 a^{20} b^{10} 
- 23040 a^{15} b^{15} - 3435 a^{10} b^{20} - 258 a^5 b^{25} - b^{30})
\textstyle\prod w^{*}_0.
\end{align*}  
\normalsize 
The covariant $S_{10}$ is (a scalar multiple of) the
determinant of the Jacobian matrix of the quadrics defining
$C_{\phi}$.  We do not know of any similar construction for $T_{30}$.
The contraction of these two covariants is $\langle S_{10}, T_{30}
\rangle = c_4^2$. In \cite[Section 8]{g1inv} we used the existence of
a such a covariant $T_{30}$ to justify our algorithm for computing the
invariants in the case of a singular genus one model.
\end{example}

In \cite[Section 7]{invenqI} we showed that the covariant $\Omega_5$
of degree $5$ in the following proposition represents the invariant
differential.

\begin{prop}
  The covariants for $Y= 
  \wedge^2 W^* \otimes S^2 W$ form a free $K[c_4,c_6]$-module of rank
  $6$ with generators in degrees $5, 15, 15, 25, 25, 35$.
\end{prop}

\begin{proof}
  The module of integer weight discrete covariants is computed as
  described in \cite{invenqI} and is found to be a free
  $K[c_4,c_6]$-module of rank $6$ with generators in degrees $5, 15,
  15, 25, 25, 35$.  We use Theorem~\ref{f1thm} to decide which of
  these are covariants.  We construct $f_1$ from $f$ by making the
  substitutions $a^5 \mapsto \prod \la_i$, $b \mapsto 1$ and
\begin{align*}
a \textstyle\sum (w_1^* \wedge w_4^*) w_0^2 & \mapsto  
\textstyle\sum \la_0 (w_1^* \wedge w_4^*) w_0^2 \\ 
\textstyle\sum (w_1^* \wedge w_4^*) w_1 w_4 & \mapsto  
\textstyle\sum (w_1^* \wedge w_4^*) w_1 w_4 \\
a^2 \textstyle\sum (w_1^* \wedge w_4^*) w_2 w_3 & \mapsto  
\textstyle\sum \la_2 \la_3 (w_1^* \wedge w_4^*) w_2 w_3 \\
a^4 \textstyle\sum (w_2^* \wedge w_3^*) w_0^2 & \mapsto  
\textstyle\sum \la_0^2 \la_1 \la_4 (w_2^* \wedge w_3^*) w_0^2 \\ 
a^3 \textstyle\sum (w_2^* \wedge w_3^*) w_1 w_4 & \mapsto  
\textstyle\sum \la_0 \la_1 \la_4 (w_2^* \wedge w_3^*) w_1 w_4 \\ 
\textstyle\sum (w_2^* \wedge w_3^*) w_2 w_3 & \mapsto  
\textstyle\sum (w_2^* \wedge w_3^*) w_2 w_3. 
\end{align*}
In this case every integer weight discrete covariant is 
a covariant. 
\end{proof}

\section{Independence of covariants}
\label{sec:indep}

Let $Y$ be a homogeneous rational representation of $\GL(V) \times \GL(W)$.

\begin{thm}
\label{thm:indep}
Assume $\charic K = 0$.  
\begin{enumerate}
\item The module of covariants $\wedge^2 V \otimes W \to Y$ is a free
  $K[c_4,c_6]$-module of rank $m = \dim Y^{H_5}$.
\item Let $F_1, \ldots, F_m$ be homogeneous covariants for $Y$.  Then
  $F_1, \ldots, F_m$ are a free basis for the module in (i) if and
  only if for each $\phi$ with $C_{\phi}$ either an elliptic normal
  quintic or rational nodal quintic, $F_1(\phi), \ldots ,F_m(\phi)$
  are linearly independent over $K$.
\end{enumerate}
\end{thm}
\begin{proof}
  (i) The fact we obtain a free module is a standard result in
  invariant theory. We have assumed $\charic K = 0$ so that $\SL(V)
  \times \SL(W)$ is linearly reductive, i.e. it has a Reynolds
  operator. Applying the Reynolds operator to the free $K[\wedge^2 V
  \otimes W]$-module of polynomial maps $\wedge^2 V \otimes W \to Y$
  shows that the covariants form a projective $K[c_4,c_6]$-module and
  hence a free $K[c_4,c_6]$-module. By Theorem~\ref{mainthm} the rank
  is the same as for the integer weight discrete covariants. We proved
  in \cite[Lemma 4.5]{invenqI} that this rank is $m$.

  (ii) Let $F_1, \ldots, F_m$ be homogeneous covariants that are a
  basis for the module in (i) and let $\phi$ be a genus one model with
  $C_{\phi}$ either an elliptic normal quintic or rational nodal
  quintic.  Suppose for a contradiction that there is a dependence
  relation
$$ \la_1 F_1(\phi) + \ldots + \la_m F_m(\phi)= 0$$
for some $\la_1, \ldots, \la_m \in K$ not all zero.  Let
\begin{equation}
\label{defd}
d= \left\{ \begin{array}{ll} 6 & \mbox{ if } c_4(\phi)=0 \\
4 & \mbox{ if } c_6(\phi)=0 \\
2 & \mbox{ otherwise. } 
\end{array} \right. 
\end{equation}
Since $\phi$ is equivalent to a Weierstrass model we see by
\cite[Proposition 4.7]{g1inv} that for every $\zeta \in \mu_d$ there
exists $g= (g_V,g_W) \in \SL(V) \times \GL(W)$ with $g\phi =\phi$ and
$\det g_W = \zeta$.  Let $F_i$ have weights $(p_i,q_i)$. Applying $g$
to the above dependence relation we obtain
\[ \zeta^{q_1} \la_1 F_1(\phi) + \ldots + \zeta^{q_m} \la_m F_m(\phi)= 0.\] 
We may therefore reduce to the case where all the $q_i$ are
congruent mod $d$. This implies by~(\ref{wtformula}) that the
degrees of the $F_i$ are congruent mod $5d$. We recall that $c_4$
and $c_6$ have degrees 20 and 30. It follows by~(\ref{defd}) that
there is a homogeneous covariant
\[ F = I_1 F_1 + \ldots + I_m F_m \] with $F(\phi)=0$ where each $I_i$
is a monomial in $c_4$ and $c_6$ and $I_i(\phi) \not=0$ for some $i$.
Then $F$ is divisible by the invariant $I$ constructed in
Lemma~\ref{lem:codim1}. Since we are assuming $F_1, \ldots, F_m$ are a
basis for the module of covariants it follows that $I$ divides $I_i$
and so $I_i(\phi) = 0$ for all $i$. This is the required
contradiction.

Conversely suppose $F_1, \ldots, F_m$ are covariants whose
specialisations at $\phi$ are linearly independent over $K$ whenever
$C_\phi$ is an elliptic normal quintic or rational nodal quintic. If
there is a relation \[ I_1 F_1 + \ldots + I_m F_m = 0 \] 
for some invariants $I_1, \ldots, I_m$ then these invariants vanish on
all non-singular models and so are identically zero by
Theorem~\ref{thm:inv}.  Thus $F_1, \ldots, F_m$ generate a free
submodule of rank $m$.  By (i) it remains to show that if
\[ I_1 F_1 + \ldots + I_m F_m = I F \] for some covariant $F$ and
invariants $I, I_1, \ldots, I_m$ then $I$ divides $I_i$ for all $i$.
We prove this by specialising to the genus one model in
Lemma~\ref{lem:van}.
\end{proof}

\begin{rem}
  (i) The result that $F_1(\phi), \ldots ,F_m(\phi)$ are linearly
  independent for $\phi$ non-singular could equally be proved using
  discrete covariants.  However this proof does not generalise to the
  case $C_\phi$ is
  a rational nodal quintic. \\
  (ii) We can remove the hypothesis $\charic K = 0$ from
  Theorem~\ref{thm:indep} (but still of course requiring $\charic K
  \not= 2,3,5$) by applying the Reynolds operator for $\SL_2(\Z/5\Z)$
  to the free $K[a^5,b^5]$-module of $\langle T \rangle$-equivariant
  maps $\Aff^2 \to Y^{H_5}$ that pass the test of Theorem~\ref{f1thm}.
  See also Remark~\ref{rem:T}.
\end{rem}

\section{Quintic covariants}
\label{sec:quintics}

We give further details of the covariants in the case $Y = S^5 W$.  We
already noted in the proof of Proposition~\ref{prop:quintics} that the
integer weight discrete covariants form a free $K[c_4,c_6]$-module of
rank 6 generated in degrees $10,10,20,20,30,30$. A basis for $Y^{H_5}$
is
\begin{align*}
\mathcal{F}_1 &= \textstyle\sum w_0^5 - 30 \prod w_0, &  
\mathcal{F}_2 &= 10 \textstyle\sum w_0^3 w_1 w_4,     &
\mathcal{F}_3 &= 10 \textstyle\sum w_0^3 w_2 w_3,     \\
\mathcal{G}_1 &= \textstyle\sum w_0^5 + 20 \prod w_0, & 
\mathcal{G}_2 &= 10 \textstyle\sum w_0 w_1^2 w_4^2,   &
\mathcal{G}_3 &= 10 \textstyle\sum w_0 w_2^2 w_3^2.
\end{align*}
In terms of this basis we have generators
\[ \begin{array}{r@{\,\,}c@{\,\,}l} \smallskip
F_{10} & = & (a^{10}-36 a^5 b^5-b^{10}) \FF_1 
+ 5 a^4 b (a^5-3 b^5) \FF_2 + 5 a b^4 (3 a^5+b^5) \FF_3, \\ 
F_{20} & = & (a^{20} + 114 a^{15} b^5 + 114 a^5 b^{15} - b^{20}) \FF_1 \\ 
& & \multicolumn{1}{c}{ 
-a^4 b (a^{15}+171 a^{10} b^5+247 a^5 b^{10}-57 b^{15}) \FF_2 }\\ 
\smallskip & & \multicolumn{1}{r}{ -a b^4 
(57 a^{15}+247 a^{10} b^5-171 a^5 b^{10}+b^{15}) \FF_3, } \\
F_{30} & = & D \big( {10} a^4 b^4 (9 a^{10} + 26 a^5 b^5 - 9 b^{10})\FF_1 \\ 
& & \multicolumn{1}{c}{ 
+a^3 (a^{15}+126 a^{10} b^5+117 a^5 b^{10}-12 b^{15}) \FF_2} \\ 
\medskip & & \multicolumn{1}{r}{
-b^3 (12 a^{15}+117 a^{10} b^5-126 a^5 b^{10}+b^{15}) \FF_3 \big), } \\ \smallskip
G_{10} & = & (a^{10}+14 a^5 b^5-b^{10}) \GG_1 + 5 a^3 b^2 (a^5+2 b^5) \GG_2 
+ 5 a^2 b^3 (2 a^5-b^5) \GG_3,  \\
G_{20} & = & (a^{20} - 136 a^{15} b^5 - 136 a^5 b^{15} - b^{20}) \GG_1
\\ & & \multicolumn{1}{c}{ -a^3 b^2 (7 a^{15} + 272 a^{10} b^5 - 221 a^5 b^{10} + 26 b^{15}) \GG_2}\\ \smallskip & & \multicolumn{1}{r}{ -a^2 b^3 (26 a^{15} + 221 a^{10} b^5 + 272 a^5 b^{10} - 7 b^{15}) \GG_3, } \\
G_{30} & = & 2  D^2 \big({10} a^3 b^3 \GG_1 + a (a^5-3 b^5) \GG_2 
-b (3 a^5+b^5) \GG_3 \big).
\end{array} 
\]

We recall from Section~\ref{sec:ex} that a discrete covariant is a
covariant if and only if the coefficient of $\sum w_0^5$ is divisible
by $a^5$. Therefore the $K[c_4,c_6]$-module of covariants for $Y=
S^5W$ has basis
\begin{equation}
\label{qbasis}
\begin{aligned}
S_{10} & = F_{10}-G_{10}, & S_{30} & = F_{30}-G_{30}, &
S_{40} & = c_6 F_{10} +c_4 F_{20}, \\
S_{20} & = F_{20}-G_{20}, & 
S'_{30} & = F_{30} + G_{30}, & S_{50} & = c_4^2 F_{10} + c_6 F_{20}. 
\end{aligned}
\end{equation} 

If we evaluate these covariants at a non-singular model $\phi$ then by
Theorem~\ref{thm:indep} we obtain a basis for the space of Heisenberg
invariant quintics.  The space of Heisenberg invariant quintics
relative to a {\em fixed} elliptic normal quintic was studied by Hulek
\cite{Hu}. We show that our basis obtained by specialising the
covariants picks out some of the quintic hypersurfaces to which Hulek
was able to attach a geometric meaning.

\begin{lem}
\label{lem:quintics}
Let $\phi \in \wedge^2 V \otimes W$ be non-singular and 
write $S_{10}, \ldots, S_{50}$ for the quintic forms
obtained by evaluating the covariants~(\ref{qbasis}) at $\phi$.
\begin{enumerate}
\item $S_{10}$ is (a scalar multiple of) the determinant of the
  Jacobian matrix of the quadrics defining $C_\phi$.
\item The Heisenberg invariant quintics vanishing on the tangent
  variety of $C_\phi$ are linear combinations of $S_{10}, S_{20},
  S'_{30}$.
\item The quintics $S_{10}, S_{20}, S_{30}$ are singular along
  $C_\phi$.
\item The quintics $S_{10}, S_{20}, S_{30}, S'_{30}, S_{40}$ vanish on
  $C_\phi$.
\end{enumerate}
\end{lem}
\begin{proof}
  For the proof we may take $\phi = u(a,b)$ a Hesse model. Let $p_0,
  \ldots, p_4$ be the equations for $C_\phi$, i.e.
  $p_i = ab  w_i^2 +b^2 w_{i+1} w_{i+4} - a^2 w_{i+2} w_{i+3}$. \\
  (i) We compute $S_{10} = 25 \det ( \partial p_i/\partial w_j)$. \\
  (ii) The tangent line to $C_\phi$ at $P= (0:a:b:-b:-a)$ also passes
  through
  \[Q = (5 a^3 b^3:0:-b (2 a^5 - b^5):-b (a^5 + 2 b^5):a (a^5 - 3
  b^5)).\] Evaluating the quintic forms at $\la P + Q$ we find that
  $S_{10}, S_{20}, S'_{30}$ vanish on the tangent line whereas
  $S_{30}, S_{40}, S_{50}$ give polynomials in $\la$ of degrees $0,2,4$. \\
  (iii) We may write these quintics as linear combinations of
  $\textstyle\sum p_0^2 w_0$, $\sum p_1 p_4 w_0$ and $\sum p_2 p_3 w_0$. \\
  (iv) We may write these quintics as linear combinations of
  $\textstyle\sum p_0 w_0^3$, $\sum p_0 w_0 w_1 w_4$, \\
  $\sum p_0 w_0 w_2 w_3$, $\sum p_0 (w_1^2 w_3 + w_2 w_4^2)$ and $\sum
  p_0 ( w_1 w_2^2 + w_3^2 w_4)$.
\end{proof}

\begin{thm}
\label{thm:hul}
Let $C = C_\phi$ be an elliptic normal quintic. Let $\Tan C$ and $\Sec
C$ be the tangent and secant varieties of $C$.  Let $F$ be the locus
of singular lines of the rank 3 quadrics containing $C$. Then
\begin{enumerate}
\item $\Sec C$ is the degree 5 hypersurface defined by $S_{10}$.
\item $\Tan C$ and $F$ are irreducible surfaces of degrees 10 and 15
  and their union is the complete intersection defined by $S_{10}$ and
  $S_{20}$.
\item The space of Heisenberg invariant quintics containing $\Tan C$
  has basis $S_{10}, S_{20}, S_{30}'$.
\item The space of Heisenberg invariant quintics containing $F$,
  equivalently that are singular along $C$, has basis $S_{10}$,
  $S_{20}$, $S_{30}$.
\item The space of Heisenberg invariant quintics containing $C$ has
  basis $S_{10}$, $S_{20}$, $S_{30}$, $S_{30}'$, $S_{40}$.
\end{enumerate}
\end{thm}
\begin{proof}
  This follows by Lemma~\ref{lem:quintics} and work of Hulek
  \cite{Hu}.
\end{proof}

\section{The covering map}
\label{sec:covmap}

We call the covariants $\wedge^2 V \otimes W \to S^d W$ 
{\em covariants of order $d$}. The action of the Heisenberg group shows
that the order must be a multiple of $5$.  By Theorem~\ref{thm:indep}
and \cite[Lemma 4.4]{invenqI}, the $K[c_4,c_6]$-modules of covariants
of orders $5, 10, 15$ have ranks $6, 41, 156$.  Fortunately we do not
need to classify all these covariants since most of them vanish on
$C_\phi$ and therefore are of no use for describing the covering
map.

\begin{lem}
\label{lemma1}
Let $C \subset \PP^{n-1}$ be an elliptic normal curve. 
Then the space of Heisenberg invariant polynomials of degree $nd$,
quotiented out by the subspace vanishing on $C$, has dimension $d$.
\end{lem}
\begin{proof}
  Let $\pi: C \to E$ be the covering map of degree $n^2$ from $C$ to
  its Jacobian $E$. Then $\pi^*(d . 0_E) \sim nd H$ where $H$ is the
  hyperplane section for $C$. So if $f_1, \ldots, f_d$ is a basis for
  the Riemann-Roch space $\CL(d. 0_E)$ then $\pi^* f_1, \ldots, \pi^*
  f_d$ are basis for the space of forms of degree $nd$ in $K[x_0,
  \ldots, x_{n-1}]/I(C)$ that are invariant under the action of
  $E[n]$. Applying the Reynold's operator for the Heisenberg group
  shows that every such form has a representative in $K[x_0, \ldots,
  x_{n-1}]$ that is itself Heisenberg invariant.
\end{proof}
 
\begin{lem}
\label{lemma2}
Let $C \subset \PP^{n-1}$ be either an elliptic normal curve or a
rational nodal curve, and let $P \in C$ be a smooth point.  Suppose
$Z,X,Y$ are homogeneous polynomials in $K[x_0, \ldots,x_{n-1}]$ of
degrees $n, 2n, 3n$ with $\ord_P(Z) = 1$, $\ord_P(X) = 0$, $\ord_P(Y)
= 0$. Then for each $d \ge 1$ the forms
\[ \{ X^i Y^j Z^k : i,k \ge 0, j \in \{0,1\}, 2i + 3j + k = d \} \]
are linearly independent in the co-ordinate ring $K[x_0, \ldots,x_{n-1}]/I(C)$.
\end{lem}
\begin{proof} 
  This is clear since $\ord_P(X^i Y^j Z^k) = k$ and the forms listed
  have distinct values of $k$.
\end{proof}

\begin{lem}
\label{lemma3}
There are covariants $Z,X,Y$ of orders $5,10,15$ and degrees $50,
110$, $165$ such that whenever $C_\phi$ is an elliptic normal quintic
or rational nodal quintic there is a smooth point $P \in C_\phi$ such
that the evaluations of $Z,X,Y$ at $\phi$ satisfy $\ord_P(Z)=1$,
$\ord_P(X)=0$, $\ord_P(Y) = 0$.
\end{lem}

\begin{proof}
  We start with the covariants $U,H : \wedge^2 V \otimes W \to
  \wedge^2 V \otimes W$ and $Q_6 : \wedge^2 V \otimes W \to S^2 V
  \otimes W$ where $U$ is the identity map and (on the Hesse family)
\begin{align*}
H & = 
-(\partial D / \partial b)  \textstyle\sum(v_1 \wedge v_4) w_0 + 
(\partial D / \partial a) \sum(v_2 \wedge v_3) w_0\\
Q_6 & = \textstyle\sum  
(5 a^3 b^3 v_0^2 + a(a^5-3b^5) v_1 v_4 - b(3 a^5+b^5) v_2 v_3) w_0. 
\end{align*} 
There are covariants $P_2, P_{12}, P_{22} : \wedge^2 V \otimes W 
\to V^* \otimes S^2 W$ where $P_2$ is the Pfaffian map~(\ref{def:P2}) and 
$P_{12}$, $P_{22}$ satisfy
\[  P_2 ( \la U + \mu H) = \la^2 P_2 + 2 \la \mu P_{12} + \mu^2 P_{22}. \]
We define covariants
$M_{30} : \wedge^2 V \otimes W \to S^5V$ and 
$N_{30} : \wedge^2 V \otimes W \to S^5V^*$
where $M_{30}=\det  Q_6 $ and $N_{30}$ is the coefficient of 
$t$ in $\det(P_2+t P_{22})$. We also define
$T_{23}$ and $T_{28}$ taking values in $V \otimes S^3W$ by
\begin{align*}
(\otimes^2 V \otimes W ) \times (V^* \otimes S^2 W) & \to V \otimes S^3W \\
(U ,P_{22}) & \mapsto T_{23} \\
(Q_6, P_{22}) & \mapsto T_{28}. 
\end{align*}
We then put
\begin{align*}
Z & = (1/2) Q_6(P_{22},P_{22}) \\
X & = (3^3/2^6) M_{30}(P_{12},P_{12},P_{12},P_{22},P_{22}) \\
Y & = (3^3/2^8) N_{30}(T_{23},T_{28},T_{28},T_{28},T_{28}).
\end{align*}
As required these are covariants of orders $5,10,15$ and
degrees $50, 110, 165$.

Suppose $C_\phi$ is a rational nodal quintic. By
Lemma~\ref{lem:orbits} we may assume that $\phi$ is as given in
Section~\ref{sec:denom}, i.e.  $\phi=u_1(0,1,1,1,1)$. Then $C_\phi$ is
parametrised by
\[ (x_0: \ldots :x_4) = (s^5-t^5:st^4:s^2t^3:-s^3t^2:-s^4t) \]
Evaluating $Z,X,Y$ at $\phi$ we find
\begin{equation}
\label{ZXYrnc}
\begin{aligned}
Z(s^5-t^5,st^4,s^2t^3,-s^3t^2,-s^4t) 
& =- 2^8 3^4 s^{10}t^{10}(s^5-t^5) \\
X(s^5-t^5,st^4,s^2t^3,-s^3t^2,-s^4t) 
& = 2^{16} 3^9 s^{20}t^{20}(s^{10}+10s^5t^5+t^{10}) \\
Y(s^5-t^5,st^4,s^2t^3,-s^3t^2,-s^4t) 
& = 2^{26} 3^{15} s^{35}t^{35}(s^5+t^5). 
\end{aligned}
\end{equation}
The conclusions of the lemma are satisfied for
$P = (0:1:1:-1:-1)$.

Now suppose $C_\phi$ is an elliptic normal quintic. Then by 
\cite[Proposition 4.1]{g1hess} 
we may assume that $\phi=u(a,b)$ is a Hesse model. 
There is a flex (i.e. hyperosculating point)
of $C_{\phi}$ at $P = (0:a:b:-b:-a)$. Evaluating $Z,X,Y$ at $\phi$ we find
\begin{equation}
\label{ZXYenc}
\begin{aligned}
Z( 0,a,b,-b,-a) & =  0 \\
X( 0,a,b,-b,-a) & =  2^{18} 3^{10} D^{10} \\
Y( 0,a,b,-b,-a) & =  -2^{27} 3^{15} D^{15} 
\end{aligned}
\end{equation}
where $D=ab(a^{10}-11a^5b^5-b^{10})$. Since $\Delta = D^5$
it is clear that $X$ and $Y$ do not vanish at $P$.  Now $C_\phi
\subset \PP^4$ is a curve of degree $5$ meeting the degree $5$
hypersurface defined by $Z$ at the $25$ flexes of $C_\phi$. So by
Bezout's theorem either $\ord_P(Z)=1$ or $Z$ vanishes identically on
$C_\phi$. To rule out the latter we write $Z$ in terms of the
basis~(\ref{qbasis}). 
Explicitly we find
\[ Z = (39/10) c_4^2 S_{10}+4 c_6 S_{20}-54 c_4 S_{30}-(198/5) c_4 S'_{30}+12 S_{50}.\]
By Theorem~\ref{thm:indep} the specialisations of 
$S_{10}, \ldots, S_{50}$ at $\phi$ are linearly independent. It follows
by Theorem~\ref{thm:hul}(v) that $Z$ does not vanish identically on $C_\phi$.
\end{proof}

\begin{lem}
\label{H90}
Let $\LL/\KK$ be a finite Galois extension with Galois group $\Gamma$.
Let $\VV$ be a finite dimensional vector space over $\LL$.  Suppose
there is an action of $\Gamma$ 
on $\VV$ satisfying $\gamma(v+w) = \gamma(v)+ \gamma(w)$ and
$\gamma(\lambda v) = \gamma(\lambda) \gamma(v)$ for all $\gamma \in
\Gamma$, $\lambda \in \LL$ and $v,w \in \VV$. Then $\dim_\KK
\VV^\Gamma = \dim_\LL \VV$.
\end{lem}
\begin{proof} A generalised form of Hilbert's Theorem 90 states that
  $H^1(\Gamma,\GL_n(\LL)) = \{1\}$. See for example \cite[Chapter X,
  Proposition 3]{SerreLF}. We fix a basis for $\VV$ over $\LL$, and
  then compare this basis with its Galois conjugates. By writing the
  resulting cocycle as a coboundary, we find a new basis for $\VV$
  over $\LL$ consisting of vectors fixed by $\Gamma$.
\end{proof}
 
\begin{lem}
\label{lem:Md}
Let $M_d$ be the $K[c_4,c_6]$-module of covariants for $Y=S^{5d}W$,
quotiented out by the submodule of covariants that vanish on the
curve.  Then $M_d$ is a free $K[c_4,c_6]$-module of rank $d$ generated
by
\[ \{ X^i Y^j Z^k : i,k \ge 0, j \in \{0,1\}, 2i + 3j + k = d \} \]
where $Z,X,Y$ are the covariants in Lemma~\ref{lemma3}.
\end{lem}

\begin{proof}
  Let $Z = (S^{5d}W)^{H_5}$ and $m = \dim Z$. We apply Lemma~\ref{H90}
  with $\KK = K(a,b)^\Gamma$, $\LL = K(a,b)$ and $\VV$ either $\UU =
  \LL \otimes_K Z$ or the subspace $\UU_0$ of forms that vanish on the
  curve defined by the generic Hesse model $u(a,b)$.  Since the action
  of $\Gamma$ on $\Aff^2$ (and hence on $\LL = K(\Aff^2)$) was defined
  so that $u : \Aff^2 \to (\wedge^2 V \otimes W)^{H_5}$ is
  $\Gamma$-equivariant, we do indeed have that $\Gamma$ acts on
  $\UU_0$. By Lemmas~\ref{lemma1} and~\ref{H90} we compute
\[ \begin{array}{l}
\dim_\KK \UU^\Gamma  =  \dim_\LL \UU = m, \\
\dim_\KK \UU_0^\Gamma  =  \dim_\LL \UU_0 = m -d. 
\end{array} \]
Thus the $K[a,b]^\Gamma$-module of discrete covariants $\Aff^2 \to Z$
has rank $m$, and the submodule of discrete covariants vanishing on
the curve has rank $m-d$.  It follows by Theorem~\ref{mainthm}, and
the proof of \cite[Lemma 4.5]{invenqI}, that the $K[c_4,c_6]$-module
of covariants $\wedge^2 V \otimes W \to S^{5d}W$ has rank $m$, and the
submodule of covariants vanishing on the curve has rank $m-d$.
Therefore $M_d$ has rank $d$.

Let $F_1, \ldots ,F_d$ be the covariants in the statement of the
lemma. Lemmas~\ref{lemma2} and~\ref{lemma3} show that if $C_\phi$ is
an elliptic normal quintic or rational nodal quintic then $F_1(\phi),
\ldots, F_d(\phi)$ are linearly independent over $K$. An argument
similar to the proof of Theorem~\ref{thm:indep}(ii) now shows that
$F_1, \ldots, F_d$ are a free basis for $M_d$.
\end{proof}

We show that the covariants $Z,X,Y$ give a formula for the covering
map.  The formula for the Jacobian was already proved in \cite{g1inv}
by a different method.

\begin{thm}
\label{thm:cov}
Let $\phi \in \wedge^2 V \otimes W$ be non-singular. Then $C_{\phi}$
has Jacobian elliptic curve $E$ with Weierstrass equation
\begin{equation}
\label{weqn}
 y^2 = x^3 - 27 c_4(\phi) x - 54 c_6(\phi) 
\end{equation}
and the covering map $C_\phi \to E$ is given by $(x,y) =
(X/Z^2,Y/Z^3)$ where $Z,X,Y$ are the evaluations at $\phi$ of the
covariants in Lemma~\ref{lemma3}.
\end{thm}

\begin{proof}
By Lemma~\ref{lem:Md} the $K[c_4,c_6]$-module $M_6$ has basis
\[ X^3, \,\, XYZ, \,\, X^2 Z^2, \,\, Y Z^3, \,\, X Z^4, \,\, Z^6. \]
Since $Z,X,Y$ have degrees $50, 110, 165$ and $c_4,c_6$ have degrees
$20, 30$ we must therefore have
\[Y^2 = \la X^3 + \mu c_4 X Z^4 + \nu c_6 Z^6\] for some $\la, \mu,
\nu \in K$. We determine these scalars by specialising to the case
$C_\phi$ is a rational nodal quintic.  Using~(\ref{ZXYrnc}) we find
$\la = 1, \mu = -27, \nu = -54$.  Thus $(x,y) = (X/Z^2, Y/Z^3)$
defines a morphism $\pi : C_{\phi} \to E$ where $E$ is the curve
defined by~(\ref{weqn}). The fibre above the point at infinity on $E$
is $C_{\phi} \cap \{ Z = 0 \}$.  By~(\ref{ZXYenc}) and Bezout's
Theorem this consists of the 25 flexes on $C_\phi$. Thus $ \deg \pi =
25$.  Since $Z,X,Y$ are covariants it is clear that $\pi$ quotients
out by the action of the Heisenberg group on $C_{\phi}$. Hence $E$ is
the Jacobian of $C_{\phi}$ and $\pi$ is the covering map.
\end{proof}

We gave algorithms for computing $Q_6$ and $H$ in \cite[Section
8]{g1inv} and \cite[Section 11]{g1hess}. So we can evaluate the
covariants $Z,X,Y$ by following the proof of Lemma~\ref{lemma3}. This
gives a practical algorithm for computing the covering map. Although
we have been working over an algebraically closed field it is clear
that Theorem~\ref{thm:cov} still holds without this assumption. We
give an example in the case $K= \Q$.

\begin{example}
\label{ex1}
Let $C \subset \PP^4$ be the elliptic normal quintic defined by the $4
\times 4$ Pfaffians of 
\small
\[ \left( \begin{array}{ccccc}
0  & 2 x_2 + 3 x_4  & 2 x_2 + x_3 + x_4 + 4 x_5 &  x_1 - x_3 + 3 x_4 - x_5  & -x_1 - x_2 - x_5 \\
 &  0 &  x_1 + 2 x_2 - x_3 - 2 x_4 + x_5 &  2 x_1 - x_2 + x_3 + 3 x_4  & -x_1 + x_2 - x_3 + x_5 \\
&  &  0 &  -2 x_2 + x_3 + x_4 + 2 x_5 &  -2 x_4 + x_5 \\
 & - &  &  0 &  x_2 + x_3 + 2 x_4 - x_5 \\
 &  &  &  &  0
\end{array} \right) \]
\normalsize
The invariants of this model are $c_4 = 21288863488$ and
$c_6 = 3106257241074688$. Our Magma function {\tt CoveringCovariants} 
evaluates the covariants of Lemma~\ref{lemma3} to give forms $Z,X,Y$.
The first of these is
\begin{align*}
 Z = & \,\,  208089517036452423241728 x_1^5 +
    481348375428118457413632 x_1^4 x_2 \\
   & -1067331097433708461809664 x_1^4 x_3 - 
     861565401032195664871424 x_1^4 x_4 \\
   & -2713065303844178403139584 x_1^4 x_5
    -1159509369215265868720128 x_1^3 x_2^2 \\
 & + \ldots +    8511800259354855263252480 x_5^5.
\end{align*}
Evaluating these forms at 
$(4013: -2384: -1616: 1388: 1021) \in C(\Q)$ we obtain
\begin{align*}
Z =& \,\, 3412377609951638022163996178720787224832, \\
X =& \,\, 12141242195111585999097107425889311253617470393872861501219624577\backslash \\
& \,\, 3843080932512669892608,  \\
Y =& \,\, 13341702475842976696719854379608150742217144829049714776419935109\backslash \\
& \,\, 10164201520123953599858396067352426339710162835468918162316066816. 
\end{align*}
The Jacobian of $C$ is the elliptic curve $E$ with Weierstrass equation
$y^2 = x^3 - 27 c_4 x - 54 c_6$ and $P = (X/Z^2,Y/Z^3) \in E(\Q)$ is a point
of canonical height $164.90718\ldots$. In fact 
$E(\Q)$ has rank $1$ and is generated by $P$.
\end{example}

\begin{rem} The elliptic curve $E$ in the above example is labelled
  17472bz1 in Cremona's tables \cite{CrTables}.  It satisfies a
  $5$-congruence with the elliptic curve $F$ labelled $17472bx2$. In
  fact $F$ has Weierstrass equation $y^2 = x (x+16)(x-26)$ and the
  genus one model in Example~\ref{ex1} may be constructed from the
  point $(-2,28) \in F(\Q)$ using visibility as described in
  \cite{invenqI}.
\end{rem}

\end{document}